\documentclass[11pt]{amsart}
\usepackage{amssymb, amsthm, amsmath, amsfonts}
\usepackage{graphics}
\usepackage{hyperref}
\usepackage[all]{xy}
\usepackage{enumerate}
\usepackage[mathscr]{eucal}
\usepackage{calrsfs, bm}
 \usepackage{amsaddr}

\theoremstyle{plain}
\newtheorem{theorem}{Theorem}[section]
\newtheorem{lemma}[theorem]{Lemma}
\newtheorem{proposition}[theorem]{Proposition}

\newtheorem{fact}[theorem]{Fact}
\numberwithin{equation}{section}

\theoremstyle{definition}

\newtheorem{definition}[theorem]{Definition}

\newtheorem{remark}[theorem]{Remark}
\newtheorem{notation}[theorem]{Notation}

\DeclareMathOperator{\Hom}{Hom}
\DeclareMathOperator{\End}{End}

\newcommand{\CC}{{\mathcal{C}}}
\newcommand{\F}{{\mathbb{F}}}
\newcommand{\FI}{{\mathrm{FI}}}
\newcommand{\GL}{{\mathrm{GL}}}
\newcommand{\kk}{{\Bbbk}}
\newcommand{\La}{{\mathit{\Lambda}}}
\newcommand{\la}{{\pmb{\bm\lambda}}}
\newcommand{\muu}{{\pmb{\bm\mu}}}
\newcommand{\nuu}{{\pmb{\bm\nu}}}
\newcommand{\PP}{{\mathcal{P}}}
\newcommand{\VI}{{\mathrm{VI}}}
\newcommand{\Z}{{\mathbb{Z}}}

\title{A representation stability theorem for VI-modules}

\author{Wee Liang Gan}
\address{Department of Mathematics, University of California, Riverside, CA 92521, USA. \\
Email: wlgan@math.ucr.edu}

\author{John Watterlond}
\address{Department of Mathematics, University of California, Riverside, CA 92521, USA. \\
Email: watterlond@math.ucr.edu}



\begin{document}

\begin{abstract}
Let VI be the category whose objects are the finite dimensional vector spaces over a finite field of order $q$ and whose morphisms are the injective linear maps. A VI-module over a ring is a functor from the category VI to the category of modules over the ring. A VI-module gives rise to a sequence of representations of the finite general linear groups. We prove that the sequence obtained from any finitely generated VI-module over an algebraically closed field of characteristic zero is representation stable - in particular, the multiplicities which appear in the irreducible decompositions eventually stabilize. We deduce as a consequence that the dimension of the representations in the sequence $\{V_n\}$ obtained from a finitely generated VI-module $V$ over a field of characteristic zero is eventually a polynomial in $q^n$. Our results are analogs of corresponding results on representation stability and polynomial growth of dimension for FI-modules (which give rise to sequences of representations of the symmetric groups) proved by Church, Ellenberg, and Farb.
\end{abstract}

\maketitle

\section{Introduction} \label{introduction}

The theory of representation stability was initiated by Church and Farb in their paper \cite{CF}. One of the main themes in this theory is to study, for an increasing chain of groups $G_0 \subset G_1 \subset G_2 \subset \cdots$, the asymptotic behavior of certain sequences 
\begin{equation} \label{sequence of representations}
\xymatrix{ V_0 \ar[r]^{\phi_0} & V_1 \ar[r]^{\phi_1} & V_2 \ar[r]^{\phi_2} & \cdots }
\end{equation}
where each $V_n$ is a representation of $G_n$, and each $\phi_n$ is a linear map. The sequence (\ref{sequence of representations}) is called a \emph{consistent sequence} if, for every non-negative integer $n$ and for every $g\in G_n$, the following diagram commutes:
\begin{equation*}
\xymatrix{ V_n \ar[r]^{\phi_n} \ar[d]_{g} & V_{n+1} \ar[d]^{g} \\ V_n \ar[r]_{\phi_n} & V_{n+1} }
\end{equation*}
(where $g$ acts on $V_{n+1}$ by considering it as an element of $G_{n+1}$).

For the family of symmetric groups $S_n$, it was discovered by Church, Ellenberg and Farb in \cite{CEF} that many interesting consistent sequences of representations of $S_n$ can be packaged into an $\FI$-module, where $\FI$ is the category of finite sets and injective maps. An $\FI$-module over a commutative ring $\kk$ is, by definition, a functor from $\FI$ to the category of $\kk$-modules; thus, an $\FI$-module $V$ gives rise to a consistent sequence (\ref{sequence of representations}) where $V_n=V(\{1,\ldots, n\})$ and $\phi_n$ is induced by the standard inclusion $\{1,\ldots, n\} \hookrightarrow \{1,\ldots, n+1\}$. One of the main results of \cite{CEF} is that the consistent sequence obtained from a finitely generated $\FI$-module $V$ over a field of characteristic zero is representation stable in the sense of \cite{CF}. 

Fix a finite field $\F_q$ of order $q$. The purpose of our present paper is to prove an analogous result for the family of finite general linear groups $\GL_n(\F_q)$. The role of the category $\FI$ will be played, in our paper, by the category $\VI$ whose objects are the finite dimensional vector spaces over $\F_q$ and whose morphisms are the injective linear maps. 

\begin{definition}
(i) A \emph{$\VI$-module} over a commutative ring $\kk$ is a functor from the category $\VI$ to the category of $\kk$-modules. 

(ii) A \emph{homomorphism} $F:U\to V$ of $\VI$-modules is a natural transformation from the functor $U$ to the functor $V$. 

(iii) Suppose $U$ and $V$ are $\VI$-modules such that $U(X)$ is a $\kk$-submodule of $V(X)$ for every object $X$ of $\VI$. We call $U$ a \emph{$\VI$-submodule} of $V$ if the collection of inclusion maps $U(X) \hookrightarrow V(X)$ defines a homomorphism $U\to V$ of $\VI$-modules.
\end{definition}

The category of $\VI$-modules over a commutative ring $\kk$ is an abelian category. 

\begin{notation}
Let $\Z_+$ be the set of non-negative integers. For each $n\in\Z_+$, we denote by $\mathbf{n}$ the object $\F_q^{\,n}$ of $\VI$.
\end{notation}

The full subcategory of $\VI$ generated by the objects $\mathbf{n}$ for all $n\in \Z_+$ is a skeleton of $\VI$. One has $\End_{\VI}(\mathbf{n}) = \GL_n(\F_q)$.

\begin{notation} \label{consistent sequence arising from V}
Suppose $V$ is a $\VI$-module. For each $n\in \Z_+$, set $V_n=V(\mathbf{n})$ and denote by $\phi_n : V_n \to V_{n+1}$ the map assigned by $V$ to the standard inclusion $\mathbf{n} \hookrightarrow \mathbf{n+1}$.
\end{notation}

The sequence (\ref{sequence of representations}) obtained from a $\VI$-module $V$ is a consistent sequence of representations of the groups $\GL_n(\F_q)$.

\begin{definition}
A $\VI$-module $V$ is \emph{generated by} a subset $S\subset \displaystyle\bigsqcup_{n\in\Z_+} V_n$ if the only $\VI$-submodule of $V$ containing $S$ is $V$;  we say that $V$ is \emph{finitely generated} if it is generated by a finite subset $S$.
\end{definition} 

In order to state our main results, let us briefly recall the parametrization of irreducible representations of the groups $\GL_n(\F_q)$ over an algebraically closed field of characteristic zero, a more detailed discussion of which will be given below in Section \ref{representations section}. 

Let $\CC_n$ be the set of cuspidal irreducible representations of $\GL_n(\F_q)$ (up to isomorphism), and let $\CC = \displaystyle\bigsqcup_{n\geqslant 1} \CC_n$. If $\rho\in \CC_n$, we set $\mathrm{d}(\rho)=n$. By a partition, we mean a non-increasing sequence of non-negative integers $(\lambda_1, \lambda_2, \ldots)$ where only finitely many terms are non-zero. If $\lambda=(\lambda_1, \lambda_2, \ldots)$ is a partition, we set $|\lambda|=\lambda_1 + \lambda_2 + \cdots$. Let $\PP$ be the set of partitions. For any function $\muu: \CC \to \PP$, let
\begin{equation*}
\| \muu \| = \sum_{\rho\in\CC} \mathrm{d}(\rho) |\muu(\rho)| . 
\end{equation*}
Then, following Zelevinsky \cite{Zelevinsky}, one has a natural parametrization of the isomorphism classes of irreducible representations of $\GL_n(\F_q)$ by functions $\muu: \CC \to \PP$ such that $\| \muu \| = n$; we shall denote by $\varphi(\muu)$ the irreducible representation of $\GL_n(\F_q)$ parametrized by $\muu$.

Let $\iota$ be the trivial representation of $\GL_1(\F_q)$. Then $\iota\in\CC_1$. Suppose $\la: \CC \to \PP$ is a function and $\la(\iota) = (\lambda_1,\lambda_2, \ldots)$. If $n$ is an integer $\geqslant \|\la\| + \lambda_1$, we define the function $\la[n] : \CC \to \PP$ with $\|\la[n]\|=n$ by 
\begin{equation*}
\la[n](\rho) = \left\{ \begin{array}{ll}
(n-\|\la\|, \lambda_1, \lambda_2, \ldots) & \mbox{ if } \rho=\iota, \\
\la(\rho) & \mbox{ if } \rho\neq\iota.
\end{array} \right.
\end{equation*}
Clearly, for each function $\muu: \CC\to\PP$ with $\|\muu\| < \infty$, there exists a unique function $\la: \CC\to\PP$ such that $\muu=\la[n]$, where $n=\|\muu\|$.

\begin{definition} \label{representation stable definition}
A consistent sequence (\ref{sequence of representations}) of representations of the groups $\GL_n(\F_q)$ over an algebraically closed field of characteristic zero is \emph{representation stable} if there exists an integer $N$ such that for each $n\geqslant N$, the following three conditions hold:
\begin{itemize}
\item[(RS1)]
{\bf Injectivity:} The map $\phi_n : V_n \longrightarrow V_{n+1}$ is injective.

\item[(RS2)]
{\bf Surjectivity:} The span of the $\GL_{n+1}(\F_q)$-orbit of $\phi_n(V_n)$ is all of $V_{n+1}$.

\item[(RS3)]
{\bf Multiplicities:}
There is a decomposition
\begin{equation*}
V_n  = \bigoplus_{\la} \varphi(\la[n])^{\oplus c(\la)}
\end{equation*}
where the multiplicities $0\leqslant c(\la) \leqslant \infty$ do not depend on $n$; in particular, for any $\la$ such that $\la[N]$ is not defined, one has $c(\la)=0$.
\end{itemize}
\end{definition}

The naming of the notion defined above is consistent with \cite[Definition 3.1]{Farb}; in \cite[Definition 3.3.2]{CEF} and \cite[Definition 2.6]{CF}, this is called \emph{uniformly} representation stable. Our main result reads:

\begin{theorem} \label{first main theorem}
Let $V$ be a $\VI$-module over an algebraically closed field of characteristic zero. Then $V$ is finitely generated if and only if the consistent sequence (\ref{sequence of representations}) obtained from $V$ is representation stable and $\dim(V_n) < \infty$ for each $n$.
\end{theorem}

After some preparation in Section \ref{stability degree and weight section}, we give the proof of Theorem \ref{first main theorem} in Section \ref{multiplicity stability section}. In Section \ref{polynomial growth of dimension section}, we deduce from property (RS3) (which holds by Theorem \ref{first main theorem}) the following result.

\begin{theorem} \label{second main theorem}
Let $V$ be a finitely generated $\VI$-module over a field of characteristic zero. Then there exists a polynomial $P\in \mathbb{Q}[T]$ and an integer $N$ such that 
\begin{equation*}
\dim (V_n) = P(q^n) \quad \mbox{ for all } n \geqslant N.
\end{equation*}
\end{theorem}

The various steps involved in the proof of Theorem \ref{first main theorem} is summarized in the following diagram.

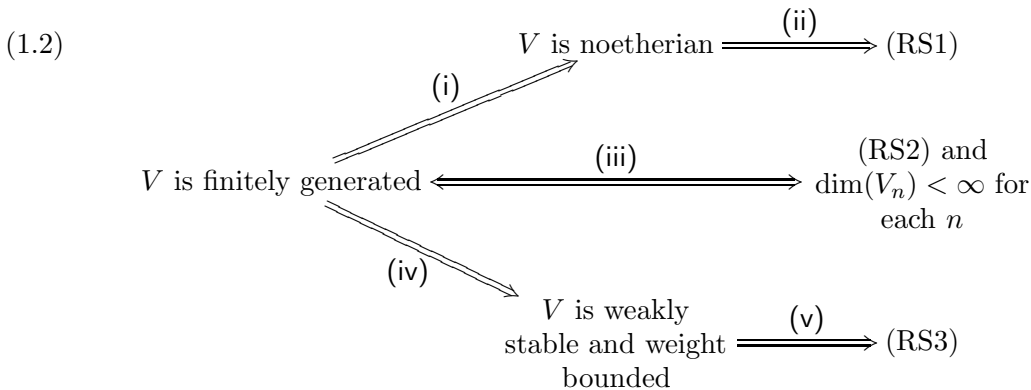
\begin{equation} \label{implications}
\xymatrix{
& \mbox{$V$ is noetherian} \ar@{=>}[r]^-{\mbox{\sf(ii)}} 
& \mbox{(RS1)}  \\ 
\mbox{$V$ is finitely generated} \ar@{=>}[ru]^-{\mbox{\sf(i)}} \ar@{<=>}[rr]^-{\mbox{\sf(iii)}} 
\ar@{=>}[rd]_-{\mbox{\sf(iv)}}
& & \parbox{3cm}{\centering (RS2) and $\dim(V_n)<\infty$ for each $n$ } \\
& \parbox{3cm}{\centering $V$ is weakly stable and weight bounded} \ar@{=>}[r]^-{\mbox{\sf(v)}}
& \mbox{(RS3)}}
\end{equation}

Implication {\sf(i)} in (\ref{implications}) was proved by the first author and Li in \cite[Theorem 3.7 and Example 3.10]{GL-Noetherian}. (More generally, Putman and Sam \cite[Theorem A]{PS}, and Sam and Snowden \cite[Corollary 8.3.3]{SS-Grobner}, proved implication {\sf(i)} over any noetherian ring.)
Implications {\sf(ii)} and {\sf(iii)} are straightforward; see \cite[Proposition 5.1 and Proposition 5.2]{GL-Noetherian} for their proofs.
Implications {\sf(iv)} and {\sf(v)} will be proved in the present paper by adapting the proofs for $\FI$-modules due to Church, Ellenberg and Farb in \cite{CEF}. A difference between our proof and theirs is that we do not try to minimize the $N$ in Definition \ref{representation stable definition}; we hope this streamlines the proof and makes it easier for the reader to grasp the essence of their argument, which is really short and nice. 

We deduce Theorem \ref{second main theorem} from Theorem \ref{first main theorem} using the hook-length formula for dimensions of irreducible representations of $\GL_n(\F_q)$. The corresponding result for $\FI$-modules was proved in \cite[Theorem 1.5]{CEF} via character polynomials. (More generally, it was proved for $\FI$-modules over a field of any characteristic in \cite[Theorem B]{CEFN}.)

Let us also mention that if $V$ is a finitely presented $\VI$-module over a commutative ring, then it is known (independently by \cite{GL-Central} and \cite{PS}) that the representations $V_n$ have an inductive description called central stability; this inductive description does not say anything about the irreducible decomposition of $V_n$ as a representation of $\GL_n(\F_q)$. In fact, the notion of central stability can be formulated in a very general setting. On the other hand, the notion of representation stability (more precisely property (RS3)) depends crucially on the sequence of groups involved. Theorems \ref{first main theorem} and \ref{second main theorem} do not follow from the results of \cite{GL-Central} or \cite{PS}.

\section{Representations of finite general linear groups} \label{representations section}

Irreducible representations of $\GL_n(\F_q)$ over an algebraically closed field of characteristic zero were first classified by Green \cite{Green}. We collect in this section the basic facts we need, following \cite{SZ} and \cite{Zelevinsky}.

\subsection{Notations}

From now on, we let $\kk$ be an algebraically closed field of characteristic zero; by a representation or a 
$\VI$-module, we mean a representation or a $\VI$-module over $\kk$.

For any finite group $G$, we write $R(G)$ for the Grothendieck group of the category of finite dimensional representations of $G$. If $\pi$ is a representation of $G$, we write $\pi^G$ for the subspace of $G$-invariants of $\pi$, and $\pi_G$ for the quotient space of $G$-coinvariants of $\pi$. 

For each $n\in \Z_+$, we set $G_n = \GL_n(\F_q)$. Let
\begin{equation*}
R = \bigoplus_{n\in \Z_+} R(G_n).
\end{equation*}

If $m,r\in \Z_+$ and $n=m+r$, let $P_{m,r}\subset G_n$ be the subgroup of matrices of the form
\begin{equation} \label{matrix p}
p = \left( \begin{array}{cc} g_{11} & g_{12} \\ 0 & g_{22} \end{array}\right),
\quad \mbox{ where } g_{11}\in G_m,\; g_{22}\in G_r. 
\end{equation}
Define the subgroups $G_{m,r}$, $H_{m,r}$, $U_{m,r}$ of $P_{m,r}$ by the conditions that, for any element $p$ of the form (\ref{matrix p}), one has:
\begin{eqnarray*}
p\in G_{m,r} & \Longleftrightarrow & g_{12} = 0, \\
p\in H_{m,r} & \Longleftrightarrow & g_{11} = 1_m, \\
p\in U_{m,r} & \Longleftrightarrow & g_{11} = 1_m \mbox{ and } g_{22} = 1_r,
\end{eqnarray*}
where, for any $m\in\Z_+$, we write $1_m$ for the identity element of $G_m$.

The composition $G_{m,r} \hookrightarrow P_{m,r} \to P_{m,r}/U_{m,r}$ is an isomorphism. If $\pi_1$ is a representation of $G_m$ and $\pi_2$ is a representation of $G_r$, we denote by $\pi_1\times \pi_2$ the representation of $G_{m+r}$ obtained from the external tensor product $\pi_1\boxtimes\pi_2$ by parabolic induction via $P_{m,r}$, that is, we regard the representation $\pi_1\boxtimes\pi_2$ of $G_{m,r}$ as a representation of $P_{m,r}/U_{m,r}$ via the isomorphism $G_{m,r} \cong P_{m,r}/U_{m,r}$, then pull it back to a representation of $P_{m,r}$ in which $U_{m,r}$ acts trivially, and let $\pi_1\times \pi_2$ be the induced representation of $\pi_1\boxtimes\pi_2$ from $P_{m,r}$ to $G_{m+r}$. This defines a multiplication on $R$. It is a well-known result of Green \cite[Lemma 2.5]{Green} that $R$ is a commutative graded ring.

\subsection{Decomposition into a tensor product}

By definition, an irreducible representation $\rho$ of $G_n$ is \emph{cuspidal} if 
\begin{equation*}
 \rho^{U_{m,n-m}} = 0 \quad \mbox{ for } m=1,\ldots n-1.
\end{equation*}
Recall that we denote by $\CC_n$ the set of cuspidal irreducible representations of $G_n$ (up to isomorphism), and write $\CC$ for $\displaystyle\bigsqcup_{n\geqslant 1} \CC_n$. For each $\rho \in \CC$, let $R(\rho)$ be the additive subgroup of $R$ generated by all $\pi \in R$ such that $\pi$ is a subrepresentation of $\rho^{\times r}$ for some $r\in\Z_+$.

\begin{fact}[{\cite[\S9]{Zelevinsky}}] \label{decomposition fact}
For each $\rho\in\CC$, the additive subgroup $R(\rho)$ of $R$ is a subring of $R$. Moreover, the multiplication map
\begin{equation} \label{tensor product decomposition}
\bigotimes_{\rho\in\CC} R(\rho) \longrightarrow R
\end{equation}
is a ring isomorphism.
\end{fact}

The tensor product in (\ref{tensor product decomposition}) is defined as the inductive limit
\begin{equation*}
\varinjlim_{\mathcal{S}} \, \bigotimes_{\rho\in \mathcal{S}} R(\rho)
\end{equation*}
where $\mathcal{S}$ runs over the finite subsets of $\CC$ partially ordered by inclusion.

\subsection{Ring of symmetric functions} 

Let 
\begin{equation*}
\La = \bigoplus_{r\in\Z_+} \La_r
\end{equation*}
be the graded ring of symmetric functions in an infinite countable set of variables with coefficients in $\Z$ (see \cite[Chapter 1]{Mac} or \cite[\S5]{Zelevinsky}). For each partition $\lambda\in\PP$, we write $s_\lambda$ for the Schur function corresponding to $\lambda$. It is well-known that, for each $r\in\Z_+$, the Schur functions $s_\lambda$ with $|\lambda|=r$ form a $\Z$-basis for $\La_r$. (Recall that for a partition $\lambda=(\lambda_1, \lambda_2, \ldots)$, we write $|\lambda|$ for $\lambda_1+\lambda_2+\cdots$.)

\begin{fact}[{\cite[\S9]{Zelevinsky}}] \label{schur functions fact}
For each $\rho\in\CC$, there is a natural isomorphism of rings 
\begin{equation*}
\La \stackrel{\sim}{\longrightarrow} R(\rho),
\end{equation*} 
denoted by 
\begin{equation*}
f \mapsto f(\rho), 
\end{equation*}
such that if $r\in\Z_+$ and $n=r\cdot\mathrm{d}(\rho)$, then the elements $s_\lambda(\rho)$ with $|\lambda|=r$ are irreducible representations of $G_n$. (Recall that $\mathrm{d}(\rho)=m$ if $\rho\in \CC_m$.)
\end{fact}

For each $n\in \Z_+$, we write $h_n$ for the $n$-th complete symmetric function, and $\iota_n$ for the trivial representation of $G_n$. Recall that $\iota=\iota_1$ and $\iota\in\CC_1$.

\begin{fact}[{\cite[\S9]{Zelevinsky}}] \label{trivial rep fact}
For each $n\in \Z_+$, one has $h_n(\iota)=\iota_n$.
\end{fact}

\subsection{Classification of irreducible representations}

Recall that for any function $\muu: \CC \to \PP$, we let $\| \muu \| = \sum_{\rho\in\CC} \mathrm{d}(\rho) |\muu(\rho)|$.

From Fact \ref{decomposition fact} and Fact \ref{schur functions fact}, one can deduce the following parametrization of the irreducible representations of the groups $G_n$.

\begin{fact}[{\cite[\S9]{Zelevinsky}}] \label{classification fact}
For each $n\in \Z_+$, the irreducible representations of the group $G_n$ are parametrized by the functions $\muu : \CC \to \PP$ such that $\| \muu \| = n$. Under this parametrization, the irreducible representation $\varphi(\muu)$ corresponding to $\muu$ is 
\begin{equation*}
\varphi(\muu) = \prod_{\rho\in\CC} s_{\muu(\rho)}(\rho).
\end{equation*}
\end{fact}

For our purposes, we do not need an explicit parametrization of the set $\CC$.

\subsection{Pieri's formula}

The Pieri's formula plays a central role in the proof of Theorem \ref{first main theorem}. We find it convenient to use the following (non-standard) notation.

\begin{notation}
For any $\lambda, \mu \in \PP$ and $r\in \Z_+$, we write $\mu \sim \lambda + r$ if the Young diagram of $\mu$ can be obtained by adding $r$ boxes to the Young diagram of $\lambda$ with no two boxes added in the same column. Similarly, we write $\lambda \sim \mu - r$ if the Young diagram of $\lambda$ can be obtained by removing $r$ boxes from the Young diagram of $\mu$ with no two boxes removed from the same column. (Thus, one has $\mu \sim \lambda + r$ if and only if $\lambda \sim \mu - r$.)
\end{notation}

\begin{fact}[Pieri's formula {\cite[Chapter 1, (5.16)]{Mac}}] \label{pieri formula}
For each $\lambda\in\PP$ and $r\in \Z_+$, one has
\begin{equation*}
s_\lambda h_r = \sum_{\mu \sim \lambda+r} s_\mu
\end{equation*}
\end{fact}

To apply Pieri's formula to representations of the groups $G_n$, we shall use the following notation.

\begin{notation}
For any functions $\la, \muu: \CC \to \PP$ and $r\in \Z_+$,  we write $\muu \sim \la + r$ if $\muu(\iota) \sim \la(\iota) + r$, and $\muu(\rho)=\la(\rho)$ for all $\rho\neq \iota$. Similarly, we write $\la \sim \muu - r$ if $\la(\iota) \sim \muu(\iota) - r$, and $\la(\rho)=\muu(\rho)$ for all $\rho\neq \iota$.
\end{notation}

\begin{lemma} \label{induction lemma}
Let $m, r\in \Z_+$. Suppose that $\la:\CC \to \PP$ is a function such that $\|\la\|=m$. Then
\begin{equation*}
\varphi(\la) \times \iota_r = \bigoplus_{\muu \sim \la + r} \varphi(\muu).
\end{equation*}
\end{lemma}
\begin{proof}
One has:
\begin{align*}
\varphi(\la) \times \iota_r &= \left(\prod_{\rho\in\CC} s_{\la(\rho)}(\rho)\right) \times h_r(\iota) &\mbox{(by Facts \ref{trivial rep fact} and \ref{classification fact})}
\\
&= \bigoplus_{\mu \sim \la(\iota)+r} \left(s_{\mu}(\iota) \times \prod_{\rho\neq\iota} s_{\la(\rho)}(\rho)\right) &\mbox{(by Facts \ref{schur functions fact} and \ref{pieri formula})} \\
&= \bigoplus_{\muu \sim \la + r} \varphi(\muu) & \mbox{(by Fact \ref{classification fact}).} 
\end{align*}
\end{proof}

Suppose $m, r\in \Z_+$ and $n=m+r$. If $\pi$ is a representation of $G_n$, then $\pi^{U_{m,r}}$ is a representation of $G_{m,r}$, and $\pi^{H_{m,r}}$ is a representation of $G_m$.

\begin{lemma} \label{restriction lemma}
Let $m, r\in \Z_+$. Suppose that $\muu: \CC \to \PP$ is a function such that $\|\muu\|=m+r$. Then
\begin{equation*}
\varphi(\muu)^{H_{m,r}} = \bigoplus_{\la \sim \muu - r} \varphi(\la).
\end{equation*}
\end{lemma}
\begin{proof}
Let $n=m+r$. For any function $\la : \CC\to\PP$ with $\|\la\|=m$, the multiplicity of $\varphi(\la)\boxtimes \iota_r$ in $\varphi(\muu)^{U_{m,r}}$ is:
\begin{eqnarray*}
\dim \Hom_{G_{m,r}} \left( \varphi(\muu)^{U_{m,r}}, \varphi(\la)\boxtimes \iota_r \right) 
&=& \dim \Hom_{G_n} \left( \varphi(\muu), \varphi(\la) \times \iota_r ) \right) \\
&=& \left\{ \begin{array}{ll} 
1 & \mbox{ if } \la \sim \muu - r,\\
0 & \mbox{ else, }
\end{array} \right.
\end{eqnarray*}
where the first equality follows from Frobenius reciprocity for parabolic induction \cite[\S8.1]{Zelevinsky}, and the second equality follows from Lemma \ref{induction lemma}.

Since $\varphi(\muu)^{H_{m,r}} = \left(\varphi(\muu)^{U_{m,r}}\right)^{G_r}$, the result follows.
\end{proof}

\section{Weak stability and weight boundedness} \label{stability degree and weight section}

In this section, we define the notions of weak stability and weight boundedness for a $\VI$-module, and prove that every finitely generated $\VI$-module is weakly stable and weight bounded.

\subsection{The VI-module {\textit{\textbf{M(m)}}}}

For each $m\in\Z_+$, define a $\VI$-module $M(m)$ by
\begin{equation*}
M(m)(-) = \kk \Hom_{\VI}(\mathbf{m}, - ),
\end{equation*}
that is, $M(m)$ is the composition of the functor $\Hom_{\VI}(\mathbf{m}, -)$ followed by the free $\kk$-module functor. It is plain (see, for example, \cite[Lemma 2.14]{GL-Noetherian}) that a $\VI$-module $V$ is finitely generated if and only if there exists a surjective homomorphism
\begin{equation*}
M(m_1) \oplus \cdots \oplus M(m_d) \longrightarrow V \quad \mbox{ for some } m_1,\ldots,m_d \in \Z_+.
\end{equation*} 

\begin{lemma}  \label{description of a free module}
Suppose $m,r\in\Z_+$ and $n=m+r$. Then
\begin{equation*}
M(m)_n = \pi \times \iota_r,
\end{equation*}
where $\pi$ is the regular representation of $G_m$.
\end{lemma}
\begin{proof}
The group $G_n$ acts transitively on $\Hom_{\VI}(\mathbf{m}, \mathbf{n})$ and the stabilizer of the standard inclusion $\mathbf{m} \hookrightarrow \mathbf{n}$ is $H_{m,r}$. The result follows from the isomorphism
\begin{equation*}
\kk[G_n/H_{m,r}] = \kk[G_n] \otimes_{\kk[P_{m,r}]} \kk[G_m].
\end{equation*}
\end{proof}

\subsection{Weak stability}
Consider a consistent sequence (\ref{sequence of representations}) with $G_n=\GL_n(\F_q)$. Suppose $m, r\in \Z_+$ and $n=m+r$. The map $\phi_n : V_n \to V_{n+1}$ descends to a map
\begin{equation} \label{map on coinvariants}
\phi_{m,r} : (V_n)_{H_{m,r}} \longrightarrow (V_{n+1})_{H_{m,r+1}}
\end{equation}
which is a homomorphism of representations of the group $G_m$.

\begin{definition} \label{weakly stable}
A consistent sequence $\{V_n, \phi_n\}$ with $G_n=\GL_n(\F_q)$ is \emph{weakly stable} if for each $m\in\Z_+$, there exists $s\in \Z_+$ such that for each $r\geqslant s$, the map $\phi_{m,r}$ of (\ref{map on coinvariants}) is an isomorphism. A $\VI$-module $V$ is \emph{weakly stable} if the consistent sequence obtained from $V$ is weakly stable (see Notation \ref{consistent sequence arising from V}).
\end{definition}

One can also define a stronger notion by requiring that the integer $s$ in Definition \ref{weakly stable} can be chosen independently of $m\in\Z_+$ (see \cite[Definition 3.1.3]{CEF}). (Unlike \cite{CEF}, we will not need to use this stronger notion.)

\begin{remark} \label{surjectivity eventually gives bijectivity}
Consider a consistent sequence $\{V_n, \phi_n\}$ with $G_n=\GL_n(\F_q)$. Supppose that $\dim(V_n) < \infty$ for every $n\in\Z_+$. Let $m\in\Z_+$, and consider the sequence of maps
\begin{equation*}
\xymatrix{  (V_m)_{H_{m,0}} \ar[r]^-{\phi_{m,0}} & (V_{m+1})_{H_{m,1}} \ar[r]^-{\phi_{m,1}} & (V_{m+2})_{H_{m,2}} \ar[r]^-{\phi_{m,2}} & \cdots.  } 
\end{equation*}
It is plain that if $\phi_{m,r}$ is surjective for all $r$ sufficiently large, then $\phi_{m,r}$ is bijective for all $r$ sufficiently large.
\end{remark}

\begin{lemma} \label{M is weakly stable}
For each $m\in\Z_+$, the $\VI$-module $M(m)$ is weakly stable.
\end{lemma}
\begin{proof}
Let $\ell \in \Z_+$. We claim that for each $r\geqslant m$, the map
\begin{equation*}
\phi_{\ell, r} : (M(m)_n)_{H_{\ell,r}} \longrightarrow (M(m)_{n+1})_{H_{\ell,r+1}} \qquad \mbox{(where $n=\ell+r$)}
\end{equation*}
is surjective. By Remark \ref{surjectivity eventually gives bijectivity}, this will imply that $M(m)$ is weakly stable.

Suppose $r\geqslant m$ and $n=\ell+r$. We write the elements of $M(m)_{n+1}$ as $(n+1)\times m$-matrices of rank $m$. For every such matrix $A$, we can multiply it on the left by an element $g$ of $H_{\ell,r+1}$ so that the last row of $gA$ is zero. The element $gA$ lies in the image of $\phi_n : M(m)_n \to M(m)_{n+1}$.
It follows that $\phi_{\ell, r}$ is surjective, as claimed.
\end{proof}

\begin{proposition} \label{fg implies weakly stable}
Let $V$ be a finitely generated $\VI$-module. Then $V$ is weakly stable.
\end{proposition}
\begin{proof}
Since $V$ is finitely generated, there exists $m_1,\ldots, m_d\in \Z_+$ and a surjective homomorphism $M\to V$ where $M=M(m_1)\oplus\cdots\oplus M(m_d)$. Let $m\in \Z_+$. By Lemma \ref{M is weakly stable}, the $\VI$-module $M$ is weakly stable. Thus, there exists $s\in \Z_+$ such that 
\begin{equation*}
\phi_{m,r} : (M_{m+r})_{H_{m,r}} \longrightarrow (M_{m+r+1})_{H_{m,r+1}}
\end{equation*}
is an isomorphism for every $r\geqslant s$. In the following commuting diagram
\begin{equation*}
\xymatrix{ (M_{m+r})_{H_{m,r}} \ar[rr]^-{\phi_{m,r}}  \ar[d] && (M_{m+r+1})_{H_{m,r+1}} \ar[d] \\
(V_{m+r})_{H_{m,r}} \ar[rr]_-{\phi_{m,r}} && (V_{m+r+1})_{H_{m,r+1}} }
\end{equation*}
the two vertical maps are surjective, so if the top horizontal map is surjective, then so is the lower horizontal map. By Remark \ref{surjectivity eventually gives bijectivity}, the result follows.
\end{proof}

\subsection{Weight boundedness}

Recall that for every function $\muu: \CC \to \PP$ with $\|\muu\| < \infty$, there is a unique function $\la : \CC \to \PP$ such that $\muu = \la[n]$, where $n=\|\muu\|$. The following definition is analogous to \cite[Definition 3.2.1]{CEF}.

\begin{definition}
A consistent sequence $\{V_n, \phi_n\}$ with $G_n=\GL_n(\F_q)$ is \emph{weight bounded} if there exists $a\in \Z_+$ such that for every $n\in\Z_+$ and every irreducible subrepresentation $\varphi(\la[n])$ of $V_n$, one has $\|\la\|\leqslant a $; we call the minimal such $a$ the \emph{weight} of the consistent sequence. A $\VI$-module $V$ is \emph{weight bounded} if the consistent sequence obtained from $V$ is weight bounded, and we define the \emph{weight} of $V$ to be the weight of its consistent sequence.
\end{definition}

\begin{proposition} \label{fg implies weight bounded}
Let $V$ be a finitely generated $\VI$-module. Then $V$ is weight bounded.
\end{proposition}
\begin{proof}
Since every quotient of a weight bounded $\VI$-module is also weight bounded, it suffices to show that for each $m\in\Z_+$, the $\VI$-module $M(m)$ is weight bounded.

Let $m, r\in \Z_+$ and $n=m+r$. Let $\muu: \CC\to\PP$ be a function with $\|\muu\|=n$. Suppose that the irreducible representation $\varphi(\muu)$ of $G_n$ is a subrepresentation of $M(m)_n$. By Lemma \ref{induction lemma} and Lemma \ref{description of a free module}, there exists a function $\nuu : \CC\to \PP$ such that $\|\nuu\|=m$ and $\muu \sim \nuu+r$. In particular, the number of columns in the Young diagram of $\muu(\iota)$ is at least $r$. If $\muu=\la[n]$, then the number of columns in the Young diagram of $\muu(\iota)$ is $n-\|\la\|$, so
\begin{equation*}
\|\la\|  \leqslant n-r=m.
\end{equation*}
\end{proof}

\begin{remark}
The \emph{generating degree} of a nonzero $\VI$-module $V$ is the smallest $m\in\Z_+\cup \{\infty\}$ such that $V$ is generated by $\displaystyle\bigsqcup_{n=0}^m V_n$. The proof of the above proposition shows that if $V$ is nonzero and has generating degree $m$, then the weight of $V$ is at most $m$. 
\end{remark}

\section{Multiplicity stability} \label{multiplicity stability section}

In this section, we complete the proof of Theorem \ref{first main theorem} by showing that every $\VI$-module which is weakly stable and weight bounded is multiplicity stable.

\subsection{Proof of multiplicity stability}

The proof of the following key proposition is an adaptation of the arguments in the proof of \cite[Proposition 3.3.3]{CEF}.

\begin{proposition} \label{multiplicity stable proposition}
Let $\{V_n, \phi_n\}$ be a consistent sequence with $G_n=\GL_n(\F_q)$. Suppose that $\{V_n, \phi_n\}$ is weakly stable and weight bounded. Then there exists an integer $N$ such that for each $n\geqslant N$, the consistent sequence $\{V_n, \phi_n\}$ satisfies condition (RS3) in Definition \ref{representation stable definition}.
\end{proposition}
\begin{proof} 
Let $a$ be the weight of the consistent sequence $\{V_n, \phi_n\}$. By weak stability, we can choose $s\in \Z_+$ such that the map $\phi_{m,r}$ of (\ref{map on coinvariants}) is an isomorphism whenever $m\leqslant a$ and $r\geqslant s$. Let $N=\max\{ a+s, 2a \}$. 

For each $n\in \Z_+$, let
\begin{equation} \label{irreducible decomposition of V_n}
V_n = \bigoplus_{\|\la\| \leqslant a} \varphi(\la[n])^{\oplus c(\la, n)} \qquad \mbox{(where $0\leqslant c(\la,n)\leqslant \infty$)}
\end{equation}
be a decomposition of $V_n$ into a direct sum of irreducible representations of $G_n$. We claim that if $\|\la\|\leqslant a$ and $n\geqslant N$, then one has $c(\la,n)=c(\la,N)$. We shall prove the claim by induction on $\|\la\|$.

Let $m\in \Z_+$ such that $m\leqslant a$. Assume that $c(\la,n)=c(\la,N)$ whenever $\|\la\|<m$ and $n\geqslant N$. (This assumption is vacuously true for $m=0$.) 

Suppose $n\geqslant N$, and set $r=n-m$. Taking $H_{m,r}$-invariants on both sides of (\ref{irreducible decomposition of V_n}), and applying Lemma \ref{restriction lemma}, we obtain
\begin{eqnarray*}
(V_n)^{H_{m,r}} &=& \bigoplus_{\|\la\| \leqslant a} \left( \varphi(\la[n])^{H_{m,r}} \right)^{\oplus c(\la, n)} \\
&=& \bigoplus_{\|\la\| \leqslant a} \left( \bigoplus_{\muu\sim\la[n]-r}  \varphi(\muu) \right)^{\oplus c(\la, n)}.
\end{eqnarray*}
Keeping in mind that the number of columns in the Young diagram of $\la[n](\iota)$ is $n-\|\la\|$, we make the following observations:
\begin{itemize}
\item
If $\|\la\|>m$, then $n-\|\la\| < n - m = r$. In this case, there is no function $\muu:\CC\to\PP$ satisfying $\muu \sim \la[n]-r$.

\item
If $\|\la\|=m$, then $n-\|\la\| = n- m =r$. In this case, the only function $\muu:\CC\to\PP$ satisfying $\muu \sim \la[n]-r$ is $\muu=\la$.
\end{itemize}
Hence, we obtain
\begin{equation} \label{isomorphism after taking invariants}
(V_n)^{H_{m,r}} = \left( \bigoplus_{\|\la\|<m} \left( \bigoplus_{\muu\sim\la[n]-r}  \varphi(\muu) \right)^{\oplus c(\la, N)} \right) \oplus \left( \bigoplus_{\|\la\|=m} \varphi(\la)^{\oplus c(\la, n)} \right).
\end{equation}

Since $r=n-m\geqslant N-a \geqslant s$, the map $\phi_{m,r}$ of (\ref{map on coinvariants}) is an isomorphism, and so we have isomorphisms of $G_m$-representations
\begin{equation} \label{left hand side}
(V_n)^{H_{m,r}} \cong (V_n)_{H_{m,r}} \cong (V_{n+1})_{H_{m,r+1}} \cong (V_{n+1})^{H_{m,r+1}},
\end{equation}
where the first isomorphism is the composition of the inclusion map $(V_n)^{H_{m,r}}\to V_n$ and the quotient map $V_n \to (V_n)_{H_{m,r}}$ (and similarly for the third isomorphism).

We claim that if $\|\la\|\leqslant a$, then 
\begin{equation} \label{first term on right hand side}
\{ \muu \mid \muu \sim \la[n] - r \} = \{ \muu \mid \muu \sim \la[n+1] - (r+1) \}.
\end{equation}
It is clear that the left hand side is a subset of the right hand side. To see that the right hand side is contained in the left hand side, we note that
\begin{equation*}
r+1 > n-m \geqslant N - a \geqslant a \geqslant \|\la\| \geqslant |\la(\iota)|,
\end{equation*}
so if the Young diagram of $\muu(\iota)$ is obtained from the Young diagram of $\la[n+1](\iota)$ by the removal of $r+1$ boxes, then one of the $r+1$ boxes removed must be from the first row.

It follows from (\ref{isomorphism after taking invariants}), (\ref{left hand side}), and (\ref{first term on right hand side}) that we have an isomorphism of $G_m$-representations
\begin{equation*}
\bigoplus_{\|\la\|=m} \varphi(\la)^{\oplus c(\la, n)} \cong \bigoplus_{\|\la\|=m} \varphi(\la)^{\oplus c(\la, n+1)}.
\end{equation*}
Hence, if $\|\la\|=m$, then $c(\la, N)= c(\la, N+1) = c(\la, N+2) = \cdots$. This completes the proof of the inductive step.
\end{proof}

\subsection{Proof of Theorem \ref{first main theorem}}

As explained in Section \ref{introduction}, we only have to prove implications {\sf(iv)} and {\sf(v)} in (\ref{implications}). 

Implication {\sf(iv)} is the combined statements of Proposition \ref{fg implies weakly stable} and Proposition \ref{fg implies weight bounded}. Implication {\sf(v)} is immediate from Proposition \ref{multiplicity stable proposition}. This completes the proof of Theorem \ref{first main theorem}.

\begin{remark}
We would like to point out that the noetherian property of $\VI$ is not used in the proofs of implications {\sf(iv)} and {\sf(v)}. (In \cite{CEF}, although the noetherian property of $\FI$ is not used in the proof of \cite[Proposition 3.3.3]{CEF}, it is used in the first paragraph in the proof of \cite[Theorem 1.13]{CEF} to show that \cite[Proposition 3.3.3]{CEF} can be applied.)

One can also deduce condition (RS1) from weak stability and weight boundedness. We give the proof below, which is adapted from the proof of \cite[Proposition 3.3.3]{CEF}. It follows that a proof of Theorem \ref{first main theorem} can be given without using the noetherian property at all. 
\end{remark}

\begin{proposition}
Let $\{V_n, \phi_n\}$ be a consistent sequence with $G_n=\GL_n(\F_q)$. Suppose that $\{V_n, \phi_n\}$ is weakly stable and weight bounded. Then there exists an integer $N$ such that for each $n\geqslant N$, the consistent sequence $\{V_n, \phi_n\}$ satisfies condition (RS1) in Definition \ref{representation stable definition}. 
\end{proposition}
\begin{proof}
Let $a$ be the weight of the consistent sequence $\{V_n, \phi_n\}$ and choose $s\in\Z_+$ such that the maps $\phi_{a,r}$ are isomorphisms for every $r\geqslant s$. Let $N=a+s$ and suppose that $n\geqslant N$.  Set $r=n-a$. Let $K_n$ be the kernel of $\phi_n : V_n \to V_{n+1}$. Then $(K_n)_{H_{a,r}}$ is contained in the kernel of $\phi_{a,r}$. Since $r\geqslant s$, the map $\phi_{a,r}$ is injective, so $(K_n)_{H_{a,r}}=0$; equivalently, one has $(K_n)^{H_{a,r}}=0$. If $K_n\neq 0$, then it contains an irreducible subrepresentation $\varphi(\la[n])$ for some $\la:\CC\to\PP$. The number of columns in the Young diagram of $\la[n](\iota)$ is $n-\|\la\|$. But $n-\|\la\| \geqslant r$, so there exists $\muu:\CC\to\PP$ such that $\muu\sim\la[n]-r$. It follows by Lemma \ref{restriction lemma} that $\varphi(\la[n])^{H_{a,r}}\neq 0$. Thus $(K_n)^{H_{a,r}}\neq 0$, a contradiction. Therefore we must have $K_n=0$ when $n\geqslant N$.
\end{proof}

\section{Dimension growth} \label{polynomial growth of dimension section}

In this section, we prove Theorem \ref{second main theorem} using Theorem \ref{first main theorem} and the hook-length formula.

\subsection{Hook-length formula}
For each $n\geqslant 1$, let
\begin{equation*}
\Phi_n(q) = \prod_{i=1}^n (q^i-1).
\end{equation*}
Suppose $\lambda=(\lambda_1, \lambda_2, \ldots)$ is a partition. We set
\begin{equation*}
\varepsilon(\lambda) = \sum_{i\geqslant 1} (i-1)\lambda_i.
\end{equation*}
We denote by $h(x)$ the hook-length at the box $x\in\lambda$ of the Young diagram of $\lambda$. Let
\begin{equation*}
\Psi_\lambda(q) = q^{\varepsilon(\lambda)} \cdot \prod_{x\in\lambda} \left( q^{h(x)} - 1 \right)^{-1}.
\end{equation*}

Let us recall the hook-length formula for the dimension of an irreducible representation of the group $G_n$. 

\begin{fact}[{\cite[Proposition 11.10]{Zelevinsky}}] \label{hook length formula}
Let $n\geqslant 1$. Let $\muu : \CC \to \PP$ be a function such that $\|\muu\|=n$. Then 
\begin{equation*}
\dim\left(\varphi(\muu)\right) = \Phi_n(q) \cdot \prod_{\rho\in\CC} \Psi_{\muu(\rho)} \left(q^{\mathrm{d} (\rho)}\right).
\end{equation*}
\end{fact}

\subsection{Proof of Theorem \ref{second main theorem}}

We now prove Theorem \ref{second main theorem}. By extension of scalars, we may assume that $V$ is a finitely generated $\VI$-module over an algebraically closed field of characteristic zero. By Theorem \ref{first main theorem}, it suffices to prove the following proposition.

\begin{proposition}
Let $m\in\Z_+$. Let $\la:\CC\to\PP$ be a function such that $\|\la\|=m$. Then there exists $N\in \Z_+$ and a polynomial $P\in \mathbb{Q}[T]$ such that
\begin{equation*}
\dim (\varphi(\la[n])) = P(q^n) \quad \mbox{ for all } n\geqslant N.
\end{equation*}
\end{proposition}
\begin{proof}
In the formula for $\dim (\varphi(\la[n]))$ given by Fact \ref{hook length formula}, the only factors which depend on $n$ are $\Phi_n(q)$ and $\Psi_{\la[n](\iota)}(q)$.

Write $\la(\iota)$ as $(\lambda_1, \lambda_2, \ldots)$, and let $N=m + \lambda_1$. Suppose $n \geqslant N$. It is clear that $\varepsilon(\la[n](\iota))$ is an integer independent of $n$. Moreover, the number of boxes in the first row of the Young diagram of $\la[n](\iota)$ is $n-m$; the hook-lengths at these boxes are:
\begin{equation*}
n-r_1,\; \ldots,\; n-r_{\lambda_1},\; n-N,\; \ldots,\; 2,\; 1.
\end{equation*}
for some $r_1 < \cdots < r_{\lambda_1} < N$. The integers $r_1, \ldots, r_{\lambda_1}$ do not depend on $n$. Let $s_1<\cdots<s_m$ be the $m$ integers such that
\begin{equation*}
\{r_1,\ldots, r_{\lambda_1}\} \sqcup \{s_1,\ldots, s_m\} = \{0, 1, \ldots, N-1\}.
\end{equation*}
It follows from Fact \ref{hook length formula} that one has
\begin{equation*} 
\dim (\varphi(\la[n])) = c (q^{n-s_1}-1) \cdots (q^{n-s_m}-1)
\end{equation*}
for some $c\in\mathbb{Q}$ which does not depend on $n$. Choosing $P\in \mathbb{Q}[T]$ to be the polynomial 
\begin{equation*}
P(T) = c (q^{-s_1} T - 1) \cdots (q^{-s_m} T - 1)
\end{equation*}
of degree $m$, we are done.
\end{proof}

\begin{remark}
From the above proof, we see that the degree of the polynomial $P$ in Theorem \ref{second main theorem} is at most the weight of $V$.
\end{remark}

\end{document}